\documentclass[12pt]{amsart}
\usepackage{amscd}     
\usepackage{amssymb}
\usepackage{amsmath, amsthm, graphics}
\usepackage{xypic}     
\LaTeXdiagrams        
\usepackage[all]{xy}
\xyoption{2cell} \UseAllTwocells \xyoption{frame} \CompileMatrices
\allowdisplaybreaks[3]
\usepackage{amsfonts}
\usepackage[normalem]{ulem}
\usepackage{hyperref}
\usepackage{tikz}
\usepackage{mathtools}
\usepackage{verbatim} 

\usepackage[margin=4cm]{geometry}

\usepackage{latexsym}
\usepackage{epsfig}
\usepackage{amsfonts}
\usepackage{enumerate}
\usepackage{times,float}
\usepackage{graphics}
\usepackage{comment}

\newtheorem{prop}{Proposition}

\newtheorem{rem}[prop]{Remark}

\newtheorem{theorem}{Theorem}[section]

\newtheorem{lemma}[theorem]{Lemma}
\newtheorem{example}[theorem]{Example}

\newtheorem{proposition}[theorem]{Proposition}

\newtheorem{definition}[theorem]{Definition}

\newtheorem{conjecture}[theorem]{Conjecture}

\theoremstyle{remark}

\theoremstyle{remark}

\newtheorem{remark}[subsection]{Remark}

\numberwithin{equation}{section}

\newcommand{\sL}{\mathcal{L}}

\newcommand{\sH}{\mathcal{H}}

\newcommand{\cO}{\mathcal{O}}

\newcommand{\bU}{\mathbb{U}}

\def\b1{{\mathbf 1}}

\begin{document}

\title{Simple Grassmannian flops}
\author{Jiun-Cheng Chen}
\address{Department of Mathematics\\ Third General Building\\ National Tsing-Hua University\\ No. 101 Sec. 2 Kuang Fu Road\\ Hsinchu, 30043\\ Taiwan}
\email{jcchen@math.nthu.edu.tw}

\author[Hsian-Hua Tseng]{Hsian-Hua Tseng}
\address{Department of Mathematics\\ Ohio State University\\ 100 Math Tower, 231 West 18th Ave.\\ Columbus, OH 43210\\ USA}
\email{hhtseng@math.ohio-state.edu}
\date{\today}

\begin{abstract}
We introduce a class of flops between projective varieties modeled on direct sums of universal subbundles of Grassmannians. We study basic properties of these flops.
\end{abstract}

\maketitle

\setcounter{tocdepth}{1}
\tableofcontents

\setcounter{section}{-1}

\section{Introduction}
We work over $\mathbb{C}$.

\subsection{Background}
A striking relationship between (classical) birational geometry and (modern) enumerative geometry is the conjectural invariance of Gromov-Witten theory when the target variety undergoes crepant birational transformations (also known as K-equivalence). Pioneered by \cite{R} and \cite{W}, this is now called {\em crepant transformation conjecture}. 

There are two main approaches to this conjecture. One approach takes advantage of knowledge about {\em descendant} Gromov-Witten theory, see \cite{CIT} and \cite{I} for further descriptions. This approach has yielded a complete solution to the crepant transformation conjecture in the toric case \cite{CIJ}, \cite{CI}.

The other approach focuses on primary/ancestor Gromov-Witten theory and is most successful to date in the case of {\em ordinary $\mathbb{P}^r$-flops}, see \cite{LLW}, \cite{ILLW}, \cite{LLW1}, \cite{LLW2}, \cite{LLQW}.

\subsection{This paper}
The purpose of this paper is to initiate the study of another class of explicit flops, in particular crepant transformation conjecture for this kind of flops.

We fix two $n$-dimensional vector spaces $V$ and $W$. Let $r< n$ be a natural number.

The central object of this paper, a {\em simple Grassmannian flop} which we introduce in Definition \ref{def:sGr_flop} below,  
\begin{equation}\label{eqn:sGr_flop}
\xymatrix{
f: X\ar@{-->}[r] & X'
}    
\end{equation}
is a global version of ``Grassmannian flops'' considered in \cite{DS}. We choose the term ``simple Grassmannian flops'' because of its clear analogy with simple $\mathbb{P}^r$-flops\footnote{More generally, we can consider the situation where Grassmannians in the construction are replaced by Grassmannian bundles over a smooth projective base, see \cite{PSW} for the local case. We refer to such a general situation as an ``ordinary Grassmannian flops'', in analogy with \cite{LLW1}, \cite{LLW2}, \cite{LLQW}.} studied in \cite{LLW}. 

Our motivating example, the {\em projective local model} of simple Grassmannian flops, is described in detail in Example \ref{example:local}.

By Lemma \ref{lem:semismall} below, simple Grassmannian flops arise from {\em semismall} contractions. In view of this, the crepant transformation correspondence (also known as invariance of Gromov-Witten theory) for simple Grassmannian flops (\ref{eqn:sGr_flop}) can be formulated as the following
\begin{conjecture}\label{conj:ctc_gr_flop}
The isomorphism in Proposition \ref{prop:coh_isom} below induces equalities between generating functions of ancestor Gromov-Witten invariants of $X$ and $X'$ after analytical continuation.     
\end{conjecture}

The main results of this paper are 
\begin{enumerate}

\item Equivalence of derived categories under simple Grassmannian flops, see Proposition \ref{prop:d-eqiuv}.

\item For projective local model of simple Grassmannian flops $\bar{X}_-\dashrightarrow \bar{X}_+$ as in Example \ref{example:local}, the induced map $(\bar{f}_+)_*\bar{f}_-^*: H^*(\bar{X}_-)\to H^*(\bar{X}_+)$ on cohomology groups give an isomorphsim between quantum cohomology rings and identifies Gromov-Witten generating functions in all genera. The result on quantum cohomology rings is given in Section \ref{sec:genus0}, the result on generating functions is given in Section \ref{sec:higher_genus}.   

In other words, we prove crepant transformation correspondence (Conjecture \ref{conj:ctc_gr_flop}) for projective local model of simple Grassmannian flops.
\end{enumerate}

The rest of this paper is organized as follows. 
Some basic properties of simple Grassmannian flops are established in Section \ref{sec:foundation}.  Our effort to study crepant transformation correspondence for simple Grassmannian flops begins with Section \ref{sec:GW}, where we treat the case of projective local models. The genus $0$ case is studied in Section \ref{sec:genus0}, and higher genus case in Section \ref{sec:higher_genus}.

\subsection{Acknowledgment}
We are very grateful to W. Donovan for corrections and comments. We thank D. Anderson for help with Lemma \ref{lem:vanishing}. H.-H. T. is supported in part by Simons foundation collaboration grant.

\section{Foundations}\label{sec:foundation}
\subsection{Local models and general definition}\label{sec:basic_property}
\begin{example}[Local models]\label{example:local} 
We fix two $n$-dimensional vector spaces $V$ and $W$. Let $r< n$ be a natural number. Denote by $S_+\to Gr(r, V)$ and $S_-\to Gr(r, W^\vee)$ the universal sub-bundles. Consider the Grassmannian flop in \cite{DS},
\begin{equation}\label{eqn:Gr_flop_local}
\xymatrix{
X_-\ar@{-->}[r] & X_+,
}    
\end{equation}
where $X_-$ is isomorphic to the {total space of $S_-\otimes V$ over $Gr(r,W^\vee)$}, $X_+$ is isomorphic to the {total space of $S_+\otimes W^\vee$ over $Gr(r,V)$}. The birational map (\ref{eqn:Gr_flop_local}) is obtained by a variation of GIT quotients, see \cite{DS} for more details. 

(\ref{eqn:Gr_flop_local}) is the flop for the morphism $$\phi_-: X_-\to X_0$$ that contracts the zero section of $S_-\otimes V$ to a point. Note the morphism $\phi_-: X_-\to X_0$ also contracts other loci. We describe the varieties $X_\pm$, $X_0$ and the morphisms $\phi_\pm : X_\pm \to X_0$. 
Denote by $$\mathsf{LinMap}^{\text{rank} \leq r}$$ the space of all linear maps from $W$ to $V$ with ranks at most $r$. We have 
\begin{equation*}
\begin{split}
&X_- = \{([W_1], A) | A \in \mathsf{LinMap}^{\text{rank} \leq r} (W,V), [W_1] \in Gr (W,r), A: W \to V \;\text{factors as} \; A: W \twoheadrightarrow W_1 \to V\},\\
&X_+ = \{([V_1], A)| A \in \mathsf{LinMap}^{\text{rank} \leq r} (W,V), [V_1] \in Gr (r,V),  \text{Im}(A) \subset V_1\},\\
&X_0=\{A| A \in \mathsf{LinMap}^{\text{rank} \leq r} (W,V)\}. 
\end{split}
\end{equation*} 
Note that we use the natural identification between the universal sub-bundle $S_ -$ over $Gr(r,W^\vee)$ and the dual of the universal quotient bundle $Q$ over $Gr(W,r)$, the Grassmannian parametrizing rank $r$ {\em quotients} of $W$.
The morphisms $\phi_\pm$ are just the morphisms which forget the first factors.  

Set $\tilde{X}:=X_-\times_{\phi_-,X_0,\phi_+}X_+$. These maps fit into a commutative diagram:  \begin{equation}\label{diagram:Gr_flop_local}
\xymatrix{
 & \tilde{X}\ar[dl]_{f_-}\ar[dr]^{f_+} & \\
X_-\ar[dr]_{\phi_-}\ar@{-->}[rr] & \, & X_+\ar[dl]^{\phi_+}\\
\, & X_0. & \, 
}    
\end{equation}
The morphisms $\phi_\pm$ are projective, but 
the varieties $X_0$ and
$X_\pm$ are not projective. 
One natural candidate for compactifying $X_-$ (resp. $X_+$) is $\mathbb{P}(X_-\oplus \mathcal{O}_{Gr(r, W^\vee)})$ (resp. $\mathbb{P}(X_+\oplus \mathcal{O}_{Gr(r,V)})$). We write 
\begin{equation*}
\bar{X}_-=\mathbb{P}(X_-\oplus \mathcal{O}_{Gr(r, W^\vee)}), \quad \bar{X}_+=\mathbb{P}(X_+\oplus \mathcal{O}_{Gr(r,V)}).
\end{equation*}


Fix bases of $V$ and $W$. A linear map $W \to V$ can be viewed as an $n \times n$ matrix. Then $X_0$ can be described as the space of $n \times n$ matrices with ranks at most $r$. 
Denote the $(i,j)$-th entry of the matrix by $x_{i,j}$. Then we have
\begin{equation*}
X_0= \{(x_{i,j})_{1 \leq i,j \leq n} \in \mathbb{C}^{n^2}| \text{all $(r+1) \times (r+1)$ minors have determinants $0$}\}.
\end{equation*}
We introduce homogeneous coordinates $(X_{i,j})_{1 \leq i \leq n, 1 \leq j \leq n}$ and $T$. Consider the homogeneous coordinate ring $$\mathbb{C}[X_{i,j},T]_{1 \leq i, j \leq n}/I$$ where $I$ is the ideal generated by determinants of all possible $(r+1) \times (r+1)$ minors. The Proj of this graded ring is a projective variety. We denote this variety by $\bar{X}_0$. It contains the open set $X_0$ defined by $\{T \neq 0 \}$. The morphisms $\phi_\pm: X_\pm \to X_0$ extend to $\bar{\phi}_\pm:\bar{X}_\pm \to \bar{X}_0$. Let $$\bar{X}_{0;\;i,j}:= \bar{X}_0 \cap \{X_{i,j} \neq 0,\, X_{i,j} \neq T\}.$$ It is easy to see that $$\bigcup _{1 \leq i,j \leq n} \bar{X}_{0;\;i,j}$$ covers the infinity divisor. Also note that $X_0 \setminus \{0\}$ can be covered by $\bigcup _{1 \leq i,j \leq n} V_{0;\;i,j}$, where $$ V_{0;\;i,j}:= \bar{X}_0 \cap \{X_{i,j} \neq 0,\, T \neq 0\}.$$ 
Set 
\begin{equation*}
\begin{split}
&\bar{X}_{-;\;i,j}:= \bar{\phi}^{-1}_-(\bar{X}_0 \cap \{X_{i,j} \neq 0,\, X_{i,j} \neq T\}),\\
&\bar{X}_{+;\;i,j}:= \bar{\phi}^{-1}_+(\bar{X}_0 \cap \{X_{i,j} \neq 0,\, X_{i,j} \neq T\}),\\
&V_{-;\;i,j}:= \bar{\phi}^{-1}_-( \bar{X}_0 \cap \{X_{i,j} \neq 0,\, T \neq 0\}),\\
&V_{+;\;i,j}:= \bar{\phi}^{-1}_+( \bar{X}_0 \cap \{X_{i,j} \neq 0,\, T \neq 0\}).
\end{split}    
\end{equation*}
For each $1 \leq i,  j \leq n$, there is an isomorphism between $\bar{X}_{0;\;i,j}$ and $V_{0;\;i,j}$. For each $1 \leq i,j \leq n$, we have  isomorphisms
\begin{equation}\label{eqn:Gr_flop_projective_ij-}
\xymatrix{
V_{-;\;i,j} \ar[d]_{\bar{\phi}_{-}|_{V_{-;\;i,j}}}\ar[r]^{\simeq} &\bar{X}_{-;\;i,j} \ar[d]^{\bar{\phi}_{-}|_{\bar{X}_{-;\;i,j}}} \\
V_{0;\;i,j} \ar[r]^{\simeq}                      &       \bar{X}_{0;\;i,j}, 
}
\end {equation}
\begin{equation}\label{eqn:Gr_flop_projective_ij+}
\xymatrix{
V_{+;\;i,j} \ar[d]_{\bar{\phi}_{+}|_{V_{+;\;i,j}}}\ar[r]^{\simeq} &\bar{X}_{+;\;i,j} \ar[d]^{\bar{\phi}_{+}|_{\bar{X}_{+;\;i,j}}} \\
V_{0;\;i,j} \ar[r]^{\simeq}                      &       \bar{X}_{0;\;i,j}. 
}
\end {equation}

The flop (\ref{eqn:Gr_flop_local}) extends to projective compactifications of $X_\pm$ over $\bar{X}_0$,
\begin{equation}\label{eqn:Gr_flop_projlocal}
\xymatrix{
\bar{X}_-:=\mathbb{P}(X_-\oplus \mathcal{O}_{Gr(r, W^\vee)})\ar@{-->}[r]_{f} & \mathbb{P}(X_+\oplus \mathcal{O}_{Gr(r,V)})=: \bar{X}_+.
}
\end{equation}
The contractions, $\bar{\phi}_-:\bar{X}_-\to \bar{X}_0\leftarrow \bar{X}_+ :\bar{\phi}_+$, and the fiber product $\tilde{\tilde{X}}:=\bar{X}_-\times_{\bar{\phi}_-,\bar{X}, \bar{\phi}_+}\bar{X}_+$ fit into a commutative diagram
\begin{equation}\label{diagram:Gr_flop_projlocal}
\xymatrix{
 & \tilde{\tilde{X}}\ar[dl]_{\bar{f}_-}\ar[dr]^{\bar{f}_+} & \\
\bar{X}_-\ar[dr]_{\bar{\phi}_-}\ar@{-->}[rr] & \, & \bar{X}_+\ar[dl]^{\bar{\phi}_+}\\
\, & \bar{X}_0. & \, 
}    
\end{equation}
(\ref{diagram:Gr_flop_projlocal}) restricts to (\ref{diagram:Gr_flop_local}).
\end{example}

Now we formulate the main definition of this paper.

\begin{definition}\label{def:sGr_flop}
A {\em simple Grassmannian flop} is a birational map $\psi $ in the following diagram
\begin{equation}\label{eqn:sGr_flopneedtochange}
\xymatrix{
X\ar[dr]_{\phi_-}\ar@{-->}[rr]^{\psi} & \, & X'\ar[dl]^{\phi_+}\\
\, & Y, & \, }    
\end{equation}
such that
\begin{enumerate}
\item $X$ and $X'$ are smooth projective varieties, $\phi_-$ and $\phi_+$ are projective morphisms;
\item $X$ contains a Zariski open subset $U_1$ which is isomorphic to a Zariski open set of $\bar{X}_{-}$ and $X'$ contains a Zariski open subset $V_1$ of $\bar{X}_{+}$ as a Zariski open set such that 
 $\psi|_{U_1}: U_1 \dashrightarrow  V_1$ is isomorphic to the restriction of the flop in Example \ref{example:local};
\item there are Zariski open subsets $U_2 \subset X$ and $V_2 \subset X'$ such that $X= U_1 \cup U_2$, $X'= V_1 \cup V_2$ and $\psi|_{U_2}: U_2 \simeq V_2$ is an isomorphism.
\end{enumerate}
\end{definition}

\subsection{Comparison with $\mathbb{P}^r$-flops}
It is known that simple $\mathbb{P}^r$-flops arise from semismall contractions. The following shows that the same is true for simple Grassmannian flops.

\begin{lemma}\label{lem:semismall}
Notations as in Example \ref{example:local}. The contraction $\bar{{\phi}}_{-}: \bar{X}_-\to \bar{X}_0$ is semismall (in the sense of intersection homology). Hence a simple Grassmannian flop $\psi$ in Definition \ref{def:sGr_flop} arises from a semismall contraction.
\end{lemma}
\begin{proof}
Following (\ref{eqn:Gr_flop_projective_ij-}), it suffices to check over $X_0$. Let $i \in \mathbb{N}$ and $X_{0}^{i} := \{x \in X_0|\,\text{dim}\, {\phi}_{-}^{-1}(x) = i \}$. These sets are locally closed subsets in $X_0$ and it makes sense to consider their dimensions. As noted in Example \ref{example:local}, $X_0$ can be identified as $\mathsf{LinMap}(W,V)^{\text{rank} \leq r}$.
Let $x=([W'], A: W \to V) \in X_{-}$, where $ W' $ is quotient vector space of $W$ of  dimension $r$ and $A:W \to V$ is a linear map with $A: W \to V$ factors as $A: W \twoheadrightarrow W' \to V$. The morphism ${\phi}_-:X_{-} \to X_0$ sends $x=([W'], A: W \twoheadrightarrow W' \to V)$ to $A: W  \twoheadrightarrow W' \to V$. Suppose that the image $A(W) \subset V$ is $k$-dimensional and denote the kernel of $A: W \to V$  by $\text{ker}(A)$, which has dimension $n-k$. The fiber over $\phi (x)$ can be identified as $Gr(\text{ker}(A), r-\text{rank} A) \simeq Gr(r-k,n-k)$, which has dimension $(r-k)(n-r)$. Therefore, the only possible $i$ such that $X_0^i$ is nonempty and has positive dimensional fiber are $i=(r-k)(n-r)$ for $k=0,1,2, \cdots, k-1.$ The dimension of $X_0^{(r-k)(n-r)}$ is $nk+ k(n-k)$. 
We have 
$$\text{dim}(X_0^{(r-k)(n-r)}) +2\times(\text{dim of the fiber})=nk+k(n-k)+2(r-k)(n-r).$$  
This is smaller than or equal to $\text{dim} X= nr+r(n-r)$ if and only if $-k^2-r^2+2rk \leq 0.$  The last inequality follows as $0 \leq k \leq r-1$. 
\end{proof}
\begin{rem}
\hfill
\begin{enumerate}
    \item The contraction $\bar{{\phi}}_{+}: \bar{X}_+\to \bar{X}_0$ is semismall by similar arguments. 
    \item Essentially the same proof was also communicated to us by W. Donovan.
\end{enumerate}
\end{rem}

Simple $\mathbb{P}^r$-flops are examples of simple $K$-equivalences in the sense of \cite{K}. The following Lemma indicates that simple Grassmannian flops and simple $\mathbb{P}^r$ flops differ in this regard.

\begin{lemma}
A simple Grassmannian flop (\ref{eqn:sGr_flop}) is {\em not} a simple K-equivalence in the sense of \cite{K}.    
\end{lemma}
\begin{proof}
For a simple Grassmannian flop, let $Z$ be the zero section of $S_-\otimes V$. Then we have $\text{dim}\,Z=\text{dim}\, Gr(r,n)=r(n-r)$. The locus $Z$ is contracted to a point, so $\text{dim}\,\bar{\phi}(Z)=0$. The rank of the normal bundle $N_{Z/X}$ is $rn$. Since $rn> r(n-r)-0-2$, if a simple Grassmannian flop were a simple K-equivalence, it fits the case \cite[Theorem 0.3 (1)]{K}. Then \cite[Theorem 0.3]{K} shows that a simple Grassmannian flop must be one of the examples constructed in \cite[Section 5]{K}. But it is clear from the descriptions in \cite[Section 5]{K} that none of those examples is a simple Grassmannian flop as in (\ref{eqn:sGr_flop}).  
\end{proof}

\subsection{Derived equivalence}
We now present one of the main results of this paper.

\begin{proposition}\label{prop:d-eqiuv}
The fibered product $\tilde{X}$ induces an equivalence of derived categories
\begin{equation*}
 R(f_{-})_{*}Lf_{+}^{*}: D^b(X')\to D^b(X).   
\end{equation*}
\end{proposition}
\begin{proof}

We first prove the compactified normal bundle case.  
\begin{equation}\label{diagram:Gr_flop_projlocal_derivedequivalencespecialcase}
\xymatrix{
\bar{X}_+\ar[dr]_{\bar{\phi}_+}\ar@{-->}[rr] & \, & \bar{X}_-\ar[dl]^{\bar{\phi}_-}\\
\, & \bar{X}_0. & \, 
}    
\end{equation}

The functor
$R(f_{-})_{*} Lf^*_{+}:D^b(X_{+}) \to D^b(X_-)$ is an equivalence by \cite[Theorems 3.2.21, 5.2.1]{BCFMV}.
The restriction of this equivalence to $D^b (V_{+;\;i,j}) \to D^b (V_{-;\;i,j})$ is also an equivalence for each $i,j$. Following (\ref{eqn:Gr_flop_projective_ij-}) and (\ref{eqn:Gr_flop_projective_ij+}), 
$ R(\bar{f}_{-;\;i,j})_{*} L\bar{f}^{*}_{+;\;i,j}: D^b(\bar{X}_{+;\;i,j}) \to D^b(\bar{X}_{-;\;i,j})$ is also an equivalence for each $i,j$. Since $\bar{X}_\pm$ are covered by open sets $X_\pm$ and $\bar{X}_{\pm;\;i,j}.$ We have  $R(\bar{f}_{-})_{*} L\bar{f}^*_+: D^b(\bar{X}_+)\to D^b(\bar{X}_-)$ is an equivalence by \cite[Proposition 3.2]{C}.

We now prove the general case. Consider the diagram:
 \begin{equation}\label{diagram:Gr_flop_local}
\xymatrix{
 & \tilde{X}\ar[dl]_{f_-}\ar[dr]^{f_+} & \\
X \ar[dr]_{\phi_-}\ar@{-->}[rr] & \, & X'\ar[dl]^{\phi_+}\\
\, & Y, & \, 
}    
\end{equation}
where $\tilde{X}:= X \times _{\phi_{-}, Y, \phi_{+}} X' $ is  the fiber product. 
Note that our definition would imply $\bar{X}_+ \setminus Z \subset X' $ and $\bar{X}_- \setminus Z \subset X $ for some closed subvariety $Z \subset X, X'$. Denote this open set in $\bar{X}_+$  by $V_1$ (resp. $\bar{X}_-$ by $U_1$). Since $U_1$ and $V_1$ can be considered as open sets in $\bar{X}_-$ and $\bar{X}_+$, the equivalence for the compactified normal bundle case implies that $ R(f_{-})_{*}Lf_{+}^{*}$ yields an equivalence when restricted to $D^b(V_1)\to D^b(U_1)$. It is clear that the restriction to $D^b(V_2) \to D^b(U_2)$ is an equivalence from the definition. Again, the desired result follows from \cite[Proposition 3.2]{C}. \end{proof}

\begin{proposition}\label{prop:coh_isom}
The fiber product $\tilde{X}:=X\times_{\bar{\phi}, \bar{X}, \bar{\phi}'} X'$ induces an isomorphism of cohomology groups:
\begin{equation*}
(f_{-})_{*}f_{+}^*: H^*(X')\to H^*(X),    
\end{equation*}
where $f_-:\tilde{X}\to X$ and $f_+:\tilde{X}\to X'$ are projection maps.
\end{proposition}
\begin{proof}
This follows from the equivalence of derived categories induced by the fiber product, see \cite[Proposition 5.33]{H}.
\end{proof}

\begin{rem}
Proposition \ref{prop:d-eqiuv} also implies that induced map between K-groups (and rational Chow groups) is an isomorphism:
\begin{equation*}
(f_{-})_{*}f_{+}^*: K(X')\overset{\sim}{\longrightarrow} K(X), \quad (f_{-})_{*}f_{+}^*: \text{CH}(X')_{\mathbb{Q}}\overset{\sim}{\longrightarrow} \text{CH}(X)_{\mathbb{Q}},
\end{equation*}
see e.g. \cite[Remark 5.25]{H} for discussions.
\end{rem}

\subsection{Discussion on embedding of normal bundles}
Our Definition \ref{def:sGr_flop} for for Grassmannian flops requires $X$ to contain an open subset of the projectivization of the vector bundle $S_{+} \otimes W^{\vee}$ over $Gr(r,V)$. This is seemingly more than what is required for simple $\mathbb{P}^r$-flops, which only require the flopping loci to be isomorphic to $\mathbb{P}^r$ with normal bundle a direct sum of copies of $\mathcal{O}_{\mathbb{P}^r}(-1)$.

In attempt to better understand the geometry of simple Grassmannian flops as in Definition \ref{def:sGr_flop}, we examine what happens if we only assume the (weaker) condition 
\begin{equation}\label{eqn:weaker_condition}
Gr(r,V)\simeq Z\subset X  \text{ with normal bundle } N_{Z/X}\simeq S_{+} \otimes W^{\vee}.
\end{equation}
We can prove the following
\begin{proposition}\label{prop:formal_nbd}
Suppose (\ref{eqn:weaker_condition}) holds. Then $X$ contains an analytic neighborhood of $Gr(r,V)\simeq Z\subset X$ which is isomorphic to an analytic open neighborhood of the zero section of the normal bundle. 
\end{proposition}

\begin{proof}
By the formal principle \cite{G}, \cite{HR}, if we can show that the formal neighborhood of $Z\subset X$ is isomorphic to the formal neighborhood of the zero section in $N_{Z/X}$, then there exists open neighborhoods of $Z\subset X$ and of zero section $Z\subset N_{Z/X}$ in $N_{Z/X}$ which are biholomorphic. Certainly, $Z\subset X$ and $Z\subset N_{Z/X}$ have isomorphic zero-th order neighborhoods, i.e. $Z$. By \cite[Proposition 1.5]{ABT} and \cite[Proposition 1.4 (i)]{ABT}, $Z\subset X$ and $Z\subset N_{Z/X}$ have isomorphic first order neighborhoods if $H^1(Z, T_Z\otimes N^\vee_{Z/X})=0$.

For $k\geq 1$. Suppose that $Z\subset X$ and $Z\subset N_{Z/X}$ have isomorphic $k$-th order neighborhoods. Then by \cite[Theorem 4.1, Proposition 2.2, Corollary 3.6]{ABT}, $Z\subset X$ and $Z\subset N_{Z/X}$ have isomorphic $(k+1)$-st order neighborhoods if $H^1(Z, T_Z\otimes \text{Sym}^{k+1}N^\vee_{Z/X})=0$ and $H^1(Z, N_{Z/X}\otimes \text{Sym}^{k+1}N^\vee_{Z/X})=0$.

By Lemma \ref{lem:vanishing} below, we conclude that $Z\subset X$ and $Z\subset N_{Z/X}$ have biholomorphic open neighborhoods.
\end{proof}

\begin{lemma}\label{lem:vanishing}
For $k\geq 1$, we have $H^1(Z, T_Z\otimes \text{Sym}^{k}N^\vee_{Z/X})=0$ and $H^1(Z, N_{Z/X}\otimes \text{Sym}^{k+1}N^\vee_{Z/X})=0$.    
\end{lemma}
\begin{proof}
Since $N^\vee_{Z/X}\simeq S^\vee\otimes W\simeq (S^\vee)^{\otimes n}$, the symmetric product $\text{Sym}^l N^\vee_{Z/X}$ is a direct sum of vector bundles of the form
$$\text{Sym}^{l_1}S^\vee\otimes\text{Sym}^{l_2}S^\vee\otimes...\otimes\text{Sym}^{l_n}S^\vee, \quad l_1+...+l_n=l.$$
Therefore $N_{Z/X}\otimes \text{Sym}^{k+1}N^\vee_{Z/X}$ is a direct sum of vector bundles of the form
$$S\otimes\text{Sym}^{l_1}S^\vee\otimes\text{Sym}^{l_2}S^\vee\otimes...\otimes\text{Sym}^{l_n}S^\vee, \quad l_1+...+l_n=k+1.$$
Thus $H^1(Z, N_{Z/X}\otimes \text{Sym}^{k+1}N^\vee_{Z/X})=0$ follows from 
\begin{equation}\label{eqn:vanishing1}
H^1(Z,S\otimes\text{Sym}^{l_1}S^\vee\otimes\text{Sym}^{l_2}S^\vee\otimes...\otimes\text{Sym}^{l_n}S^\vee)=0.
\end{equation}

The universal sequence on $Gr(r,n)$, $$0\to S\to \mathcal{V}\to Q\to 0,$$
where $\mathcal{V}$ is the trivial bundle of rank $n$, gives the following short exact sequence
\begin{equation}
0\to S\otimes S^\vee\to \mathcal{V}\otimes S^\vee\to Q\otimes S^\vee\simeq T_{Gr(r,n)}\to 0.    
\end{equation}
The associated long exact sequence implies that $H^1(Z, T_Z\otimes \text{Sym}^{k}N^\vee_{Z/X})=0$ follows from $H^1(Z, S\otimes S^\vee\otimes \text{Sym}^kN^\vee_{Z/X})=0$ and $H^1(Z,\mathcal{V}\otimes S^\vee\otimes \text{Sym}^k N^\vee_{Z/X})=0$, which in turn follow from 
\begin{equation}\label{eqn:vanishing2}
\begin{split}
&H^1(Z, S\otimes S^\vee\otimes\text{Sym}^{l_1}S^\vee\otimes\text{Sym}^{l_2}S^\vee\otimes...\otimes\text{Sym}^{l_n}S^\vee)=0,\\
&H^1(Z, S^\vee\otimes \text{Sym}^{l_1}S^\vee\otimes\text{Sym}^{l_2}S^\vee\otimes...\otimes\text{Sym}^{l_n}S^\vee)=0.
\end{split}    
\end{equation}
It remains to show (\ref{eqn:vanishing1}) and (\ref{eqn:vanishing2}). These are consequences of Borel-Weil-Bott theorem. We use the results of \cite{EF}. Recall from \cite[Section 2]{EF}, $Z\simeq Gr(r,n)=U_{r,n}/GL_r$ where $U_{r,n}\subset \mathbb{C}^{nr}$ is the space of $n\times r$ complex matrices of rank $r$ and $GL_r$ acts by right multiplication. A $GL_r$ representation induces a vector bundle on $Gr(r,n)$. In particular, the defining representation of $GL_r$ gives the vector bundle $S^\vee$. By e.g. \cite[Section 15.5]{FH}, irreducible representations of $GL_r$ correspond to non-increasing sequences $\lambda_1\geq \lambda_2\geq...\geq \lambda_r$ of integers. 

The vector bundles 
\begin{equation*}
S^\vee\otimes\text{Sym}^{l_1}S^\vee\otimes\text{Sym}^{l_2}S^\vee\otimes...\otimes\text{Sym}^{l_n}S^\vee
\end{equation*}
correspond to direct sums of irreducible $GL_r$ representations indexed by integer sequences $\lambda_1\geq \lambda_2\geq...\geq \lambda_r$ with $\lambda_r\geq 0$. By \cite[Theorem 2.2 (a)]{EF}, we have 
\begin{equation*}
H^1(Z, S^\vee\otimes \text{Sym}^{l_1}S^\vee\otimes\text{Sym}^{l_2}S^\vee\otimes...\otimes\text{Sym}^{l_n}S^\vee)=0.
\end{equation*}
Since $S\simeq R\otimes D$ where $R$ is the vector bundle corresponding to the sequence $1,1,...,1,0$ and $D$ is the line bundle corresponding to the determinantal representation of $GL_r$ (which corresponds to the sequence $-1,...,-1$), the vector bundles 
\begin{equation*}
\begin{split}
&S\otimes\text{Sym}^{l_1}S^\vee\otimes\text{Sym}^{l_2}S^\vee\otimes...\otimes\text{Sym}^{l_n}S^\vee\\
&S\otimes S^\vee\otimes\text{Sym}^{l_1}S^\vee\otimes\text{Sym}^{l_2}S^\vee\otimes...\otimes\text{Sym}^{l_n}S^\vee
\end{split}
\end{equation*}
correspond to direct sums of irreducible $GL_r$ representations indexed by integer sequences $\lambda_1\geq \lambda_2\geq...\geq \lambda_r$ with $\lambda_r\geq -1$. If $\lambda_r\geq 0$, we have 
\begin{equation}\label{eqn:vanishing3}
\begin{split}
&H^1(Z,S\otimes\text{Sym}^{l_1}S^\vee\otimes\text{Sym}^{l_2}S^\vee\otimes...\otimes\text{Sym}^{l_n}S^\vee)=0\\
&H^1(Z,S\otimes S^\vee\otimes\text{Sym}^{l_1}S^\vee\otimes\text{Sym}^{l_2}S^\vee\otimes...\otimes\text{Sym}^{l_n}S^\vee)=0
\end{split}    
\end{equation}
as desired. If $\lambda_r=-1$, then examining the proof of \cite[Theorem 2.2 (b)]{EF}, the desired vanishing (\ref{eqn:vanishing3}) is still valid for all $r,n$ with $n\geq 1+r$ (which is our case).
\end{proof}

\section{Gromov-Witten theory of projective local models}\label{sec:GW}
In this Section we discuss crepant transformation correspondence for simple Grassmannian flops in the basic example of projective local models (\ref{eqn:Gr_flop_projlocal}). Our arguments build on results of \cite{LSW}, \cite{PSW}, \cite{IK}. We omit an introduction to Gromov-Witten theory and refer to these papers for setups and notations needed here.

\subsection{Genus zero}\label{sec:genus0}
We begin with the results in \cite{LSW}, \cite{PSW} on identifying genus $0$ Gromov-Witten theories of local models (\ref{eqn:Gr_flop_local}), formulated using Givental's symplectic space formalism: there is an explicitly constructed symplectic transformation 
\begin{equation}\label{eqn:U_open1}
\bU: \sH^{X_-}\to \sH^{X_+} 
\end{equation}
that identifies Givental's Lagrangian cone for $X_\pm$ (by identifying their big $I$-functions):
\begin{equation}\label{eqn:U_open2}
\bU(\sL^{X_-})=\sL^{X_+}, \quad \bU I^{X_-}=I^{X_+}.    
\end{equation}
Here, since $X_\pm$ are non-compact, we consider their Gromov-Witten theories equivariant with respect to natural torus actions given by the product of the $T=(\mathbb{C}^*)^n$-action on Grassmannian and $T$-action on $S^{\oplus n}$.

By genus $0$ quantum Riemann-Roch theorem \cite{CG}, (\ref{eqn:U_open1}) yields a symplectic transformation (we abuse notations here) 
\begin{equation}\label{eqn:U+O1}
\bU:\sH^{X_-\oplus \cO}\to \sH^{X_+\oplus \cO},
\end{equation}
such that
\begin{equation}\label{eqn:U+O2}
\bU(\sL^{X_-\oplus \cO})=\sL^{X_+\oplus \cO},\quad \bU I^{X_-\oplus \cO}=I^{X_+\oplus \cO}.
\end{equation}
Here $X_\pm\oplus \cO$ carry the $T\times T\times \mathbb{C}^*$-actions given by the $T\times T$-actions on $X_\pm$ and the fiber-scaling $\mathbb{C}^*$-action on $\cO$. Adding the extra $\cO$ corresponds to twisting Gromov-Witten theory of $X_\pm$ by $\cO$ and $e_{\mathbb{C}^*}^{-1}$. Since the Hodge bundle in genus $0$ is rank $0$, this mostly has no effects. Also, we may define $I$-functions $I^{X_\pm\oplus \cO}$ using the recipe in \cite{CCIT}.

We now proceed to study $\bar{X}_\pm=\mathbb{P}(X_\pm\oplus \cO)$ as in (\ref{eqn:Gr_flop_projlocal}). For this, we forget the $T$-action on Grassmannian and restrict the $T\times \mathbb{C}^*$-action on fibers of vector bundles $S^{\oplus n}\oplus \cO$ to the diagonal $\mathbb{C}^*\subset T\times \mathbb{C}^*$. By the results of \cite{IK}, (\ref{eqn:U+O1}) yields a symplectic transformation 
\begin{equation}\label{eqn:U_projlocal1}
\bar{\bU}:\sH^{\bar{X}_-}\to \sH^{\bar{X}_+},
\end{equation}
such that
\begin{equation}\label{eqn:U_projlocal2}
\bar{\bU}(\sL^{\bar{X}_-})=\sL^{\bar{X}_+},\quad \bar{\bU} I^{\bar{X}_-}=I^{\bar{X}_+}.
\end{equation}
Here, the $I$-functions $I^{\bar{X}_\pm}$ can be constructed from $I^{X_\pm\oplus \cO}$ using \cite[Theorem 3.3]{IK}. The symplectic transformation $\bar{\bU}$ is defined so that the following holds
\begin{equation}\label{eqn:U_and_barU}
\bar{\bU}(\hat{J})=\widehat{\bU(J)}.    
\end{equation}
Here the Fourier transform $\widehat{(-)}$ defined in \cite[Section 4.2]{IK} is
\begin{equation}
J(\lambda)\mapsto \hat{J}(q)=\sum_{k\in \mathbb{Z}}\kappa(\mathcal{S}^{-k}J)q^k.    
\end{equation}
Therefore, $\bar{\bU}$ is explicitly given as follows. For $g(q)=\sum_{k\in \mathbb{Z}}g_kq^k$, we have
\begin{equation}
\bar{\bU}(g(q))=\sum_{k\in\mathbb{Z}} \bar{\bU}(g_k)q^k, \quad \bar{\bU}(g_k)=\kappa_+\mathcal{S}^{-k}\bU\mathcal{S}^k\kappa_-^{-1}(g_k).    
\end{equation}
Here $\kappa_\pm: H^*_{\mathbb{C}^*}(X_\pm\oplus \cO)\to H^*(\bar{X}_\pm)$ are the Kirwan maps (see \cite[Section 1.3]{IK}) and $\kappa_\pm^{-1}$ are defined to send $p=c_1(\cO(1))$ to the equivariant parameter $\lambda$. $\mathcal{S}$ is the shift operator \cite[Equation (1.1)]{IK}.

Since $\bU$ in (\ref{eqn:U+O1}) is symplectic, it follows from \cite[Proposition 4.4]{IK} and a direct calculation that $\bar{\bU}$ in (\ref{eqn:U_projlocal1}) is also symplectic. By Lemma \ref{lem:semismall} and \cite[Remark 3.19, Theorem 3.22]{I}, $\bar{\bU}$ preserves the opposite subspace,
\begin{equation}\label{eqn:opposite_subspace}
\bar{\bU}(\sH^{\bar{X}_-}_-)=\sH^{\bar{X}_+}_-.    
\end{equation}
Because of (\ref{eqn:opposite_subspace}), applying Birkhoff factorization to derivatives of the second equation in (\ref{eqn:U_projlocal2}), we obtain
\begin{equation}\label{eqn:U_projlocal3}
\bar{\bU}(S^{\bar{X}_-})^*=(S^{\bar{X}_+})^*,    
\end{equation}
where $S^{\bar{X}_\pm}$ is the generating function of genus $0$, $1$-point descendant invariants of $\bar{X}_\pm$. 

Now we calculate the induced isomorphism on cohomology
\begin{equation}\label{eqn:bbU_H}
\bar{\bU}_H: H^*(\bar{X}_-)\simeq z\sH^{\bar{X}_-}_-/\sH^{\bar{X}_-}_-\overset{\bar{\bU}}{\longrightarrow} z\sH^{\bar{X}_+}_-/\sH^{\bar{X}_+}_-\simeq H^*(\bar{X}_+).
\end{equation}
Consider the diagram
\begin{equation}\label{diagram:FM}
\xymatrix{
K_{eq}(X_-)\ar[r]^{FM}\ar[d]_{\Psi_-} & K_{eq}(X_+)\ar[d]^{\Psi_+}\\
\sH^{X_-}\ar[r]^{\bU}\ar[d] & \sH^{X_+}\ar[d]\\
\sH^{X_-\oplus \cO}\ar[r]^{\bU}\ar[d]_{\widehat{(-)}} & \sH^{X_+\oplus \cO}\ar[d]^{\widehat{(-)}}\\
\sH^{\bar{X}_-}\ar[r]^{\bar{\bU}} & \sH^{\bar{X}_+}.
}    
\end{equation}
Here the upper square is commutative because of \cite{LSW}, \cite{PSW}. The middle square is commutative by definition. The lower square is commutative because of \cite[Theorem 4.2]{IK}. 

Using the definition of Fourier transform $\widehat{(-)}$, we see that $\widehat{(-)}$ sends $\sH^{X_-}_-$ to $\sH^{X_+}_-$ and $\widehat{z^{>0}}$ contains $z^{>0}$ terms. It follows that 
\begin{equation}
\bU(\sH^{X_-}_-)=\sH^{X_+}_-.    
\end{equation}
Define
\begin{equation}\label{eqn:bU_H}
{\bU}_H: H_{eq}^*({X}_-)\simeq z\sH^{{X}_-}_-/\sH^{{X}_-}_-\overset{{\bU}}{\longrightarrow} z\sH^{{X}_+}_-/\sH^{{X}_+}_-\simeq H_{eq}^*({X}_+).
\end{equation}

The commutativity of the upper square in (\ref{diagram:FM}) is 
\begin{equation}\label{eqn:FM_compatible}
\bU\circ \Psi_-=\Psi_+\circ FM.    
\end{equation}
We examine (\ref{eqn:FM_compatible}) evaluated at $\alpha\in K_{eq}(X_-)$:
\begin{equation}\label{eqn:FM_compatible2}
\begin{split}
&\bU\left(z^{-\mu_-}z^{\rho_-}(\Gamma_{X_-}\cup (2\pi\sqrt{-1})^{\frac{\text{deg}_{X_-}}{2}}ch(\alpha)) \right)\\
=&z^{-\mu_+}z^{\rho_+}(\Gamma_{X_+}\cup (2\pi\sqrt{-1})^{\frac{\text{deg}_{X_+}}{2}}ch(FM(\alpha)))\\
=&z^{-\mu_+}z^{\rho_+}(\Gamma_{X_+}\cup (2\pi\sqrt{-1})^{\frac{\text{deg}_{X_+}}{2}}(f_+)_*(ch(f_-^*\alpha)Td(T_{f_+}))).
\end{split}
\end{equation}
The action of the operators $z^{-\mu_\pm}$ is given by
\begin{equation}
z^{-\mu_\pm}(c)=z^{\frac{1}{2}\text{dim} X_\pm}\cdot z^{-\frac{1}{2}\text{deg}_{X_\pm}(c)}\cdot c.    
\end{equation} 
The action of the operator $z^{-\frac{1}{2}\text{deg}_{X_\pm}(c)}$ introduces non-positive powers of $z$. The operators $z^{\rho_\pm}$ are defined to be multiplications by $\sum_{k\geq 0}\frac{(\log z)^k\rho_\pm^k}{k!}$. After removing the common factor $z^{\frac{1}{2}\text{dim} X_\pm}$ from both sides of (\ref{eqn:FM_compatible2}), we see that comparing the terms of highest order in $z$ yields $\bU_H ch_0(\alpha)=(f_+)_*f_-^*(ch_0(\alpha))$. 

The term $\bU_H ch_1(\alpha)$ is the unique term on the left-hand side of (\ref{eqn:FM_compatible2}) proportional to $(2\pi\sqrt{-1})^1\cdot z^{-1}$. On the right-hand side of (\ref{eqn:FM_compatible2}), the term proportional to $(2\pi\sqrt{-1})^1\cdot z^{-1}$ is $(f_+)_*f_-^*(ch_1(\alpha))$. Therefore we obtain $\bU_H ch_1(\alpha)=(f_+)_*f_-^*(ch_1(\alpha))$. Comparing terms proportional to $(2\pi\sqrt{-1})^k\cdot z^{-k}$, we get
\begin{equation}
\bU_H ch_k(\alpha)=(f_+)_*f_-^*(ch_k(\alpha)), \quad k\geq 0.
\end{equation}
This shows that $\bU$ induces 
\begin{equation}\label{eqn:bU_H_ans}
\bU_H=(f_+)_*f_-^*
\end{equation} 
on equivariant cohomology of $X_\pm$ (and $X_\pm\oplus \cO$). 

By (\ref{eqn:U_and_barU}), the induced map $\bar{\bU}_H$ in (\ref{eqn:bbU_H}) satisfies
\begin{equation}\label{eqn:UbU}
\kappa_+\circ \bU_H=\bar{\bU}_H\circ \kappa_-.    
\end{equation}
The diagram (\ref{diagram:Gr_flop_local}) give the following two cartesian squares,
\begin{equation*}
\xymatrix{
(X_-\oplus\cO)^o\ar[d]_{\iota_-} & (\tilde{X}\oplus\cO)^o\ar[d]^{\tilde{\iota}}\ar[l]_{f_-}\ar[r]^{f_+} &(X_+\oplus\cO)^o\ar[d]^{\iota_+}\\
X_-\oplus\cO & \tilde{X}\oplus \cO\ar[l]^{f_-}\ar[r]_{f_+} &X_+\oplus \cO,
}    
\end{equation*}
where $(-)^o$ denotes the complement of the zero section. We have 
\begin{equation}\label{eqn:Kirwan_compatible}
\tilde{\iota}^*f_-^*=f_-^*\iota_-^*, \quad \iota_+^*(f_+)_*=(f_+)_*\tilde{\iota}^*.    
\end{equation}
Since the Kirwan maps $\kappa_\pm$ are given by $\iota_\pm^*$ in equivariant cohomology, (\ref{eqn:bU_H_ans}), (\ref{eqn:UbU}), and (\ref{eqn:Kirwan_compatible}) imply
\begin{equation}\label{eqn:FM_coh}
\bar{\bU}_H=(\bar{f}_+)_*\bar{f}_-^*: H^*(\bar{X}_-)\to H^*(\bar{X}_+).
\end{equation}

It follows from the whole discussion above that (\ref{eqn:FM_coh}) is an isomorphism of quantum cohomology rings.

\subsection{Higher genus}\label{sec:higher_genus}
By e.g. \cite{BCFK}, the genus $0$ Gromov-Witten theory of Grassmannian is semisimple. By \cite[Corollary 1.8]{IK}, genus $0$ Gromov-Witten theories of $\bar{X}_\pm=\mathbb{P}(X_\pm\oplus \cO)$ is semisimple. Then we may apply the arguments in \cite{ILLW} to show that (\ref{eqn:FM_coh}) induces equalities on higher genus ancestor potentials of $\bar{X}_\pm$. More precisely, the discussions in Section \ref{sec:genus0} show that (\ref{eqn:FM_coh}) induces an isomorphism of genus $0$ Gromov-Witten theories of $\bar{X}_\pm$ at the level of Frobenius manifolds. Also, $\bar{\bU}$ identifies the $J$-functions of $\bar{X}_\pm$. Then, we can use Givental-Teleman reconstruction in the manner of \cite{ILLW} to identify higher genus theories.

With additional care (especially on the analytical continuation), it should be possible to show that the {\em descendant} Gromov-Witten theory of $\bar{X}_\pm$ are equated by applying the quantization of $\bar{\bU}$, by following the approach of \cite{CI}.

\end{document}